\documentclass[12pt]{article}

\usepackage[top=1.2in, bottom=1.2in, left=1.25in, right=1.25in]{geometry}

\usepackage[textwidth=1.15in, textsize=footnotesize]{todonotes}
\presetkeys{todonotes}{fancyline, color=white}{}
\usepackage{microtype, mathtools}
\usepackage{diagbox}

\usepackage{xparse}
\ExplSyntaxOn
\NewDocumentCommand{\rvect}{m}
 {
  \seq_set_split:Nnn \l_tmpa_seq { , } { #1 }
  \begin{bmatrix}
  \seq_use:Nn \l_tmpa_seq { & }
  \end{bmatrix}
 }
\ExplSyntaxOff

\usepackage{amsmath,amssymb}
\usepackage{amsthm}
\newtheorem{definition}{Definition}

\newtheorem{theorem}{Theorem}
\newtheorem{proposition}{Proposition}
\newtheorem{corollary}{Corollary}

\usepackage{thmtools}
\declaretheorem[style=definition,qed=$\triangle$]{example}

\theoremstyle{definition}

\def\boxitat#1#2{\vbox             
     {\hrule\hbox{\vrule\kern#1%
     \vbox{\kern #1\hbox{#2}\kern#1}\kern#1\vrule}\hrule}}
\def\enclose#1{\boxitat{0pt}{#1}}

\def\qedspace{\null}
\def\qedblack{\qedspace                    
  \lower.6pt\hbox{\vrule height7pt width 5pt}}
\def\qedwhite{\qedspace                    
  \lower.6pt\enclose{%
        \hbox{\vrule height7pt width 0pt\hskip6pt}}}

\def\beq{\begin{equation}}
\def\eeq{\end{equation}}

\usepackage{amssymb} 
\usepackage{csquotes}

\newcommand{\disjun}{\mathbin{\dot{\cup}}}

\def\eps{{\varepsilon}}
\def\si{{\sigma}}
\def\om{{\omega}}

\def\sW{{\cal W}}
\def\sR{{\cal R}}
\def\sT{{\cal F}}

\def\hC{{\hat C}}

\def\bx{{\bf x}}

\def\bu{{\bf u}}
\def\bv{{\bf v}}
\def\bz{{\bf z}}
\def\bn{{\bf n}}
\def\bj{{\bf j}}
\def\bs{{\bf s}}

\def\bg{{\bf g}}
\def\bh{{\bf h}}

\def\fq{{\mathbb F}_q}
\def\fp{{\mathbb F}_p}
\def\Zk{{{\mathbb Z}_k}}
\def\Zp{{{\mathbb Z}_p}}
\def\mZ{{\mathbb Z}}

\def\term{\mbox{\boldmath $\beta$}}
\def\start{\mbox{\boldmath $\alpha$}}

\def\sR{{\cal R}}
\def\sT{{\cal T}}

\def\SEQ{\textsc{Seq}}
\def\spn{{\sigma}}

\begin{document}

\title{Counting  compositions over finite abelian groups}

\author{
  Zhicheng Gao\thanks{
    Email: \texttt{zgao@math.carleton.ca}. This work is supported in part by
    NSERC.},\,
  Andrew MacFie\thanks{
    Email: \texttt{andrewmacfie@carleton.ca}.},\,
  Qiang Wang\thanks{
    Email: \texttt{wang@math.carleton.ca}.
    This work is supported in part by NSERC.}\\
\small School of Mathematics and Statistics\\[-0.8ex]
\small Carleton University\\[-0.8ex]
\small Ottawa, Ontario, Canada K1S5B6\\
}


\maketitle

\smallskip

\centerline{\small Mathematics Subject Classification: 05A15, 05A16}

\begin{abstract}
We find the number of compositions over finite abelian groups under two types
of restrictions: (i) each part belongs to a given subset and (ii) small
runs of consecutive parts must have given properties.
Waring's problem over finite fields can be converted to type~(i) compositions,
whereas Carlitz and \enquote{locally Mullen} compositions can be formulated as
type~(ii) compositions.
We use the {multisection formula} to translate the problem from
integers to group elements, the transfer matrix method to do exact counting,
and finally the Perron-Frobenius theorem to derive asymptotics.
We also exhibit bijections involving certain restricted classes of
compositions.

\end{abstract}

\section{Introduction}

Let $n$ and $m$ be positive integers.
A \emph{composition} of $n$ is a sequence of positive integers whose sum
is $n$.
An \emph{$m$-composition} is a composition consisting of $m$ terms, also called
parts.
It is well known that there is a bijection between $m$-compositions of $n$
and $(m-1)$-subsets of $\{1, 2, \ldots, n-1\}$ and thus there are
$\binom{n-1}{m-1}$ $m$-compositions of $n$ and $2^{n-1}$ compositions of $n$.
A \emph{weak composition} is the same as a composition except terms equal to
$0$ are allowed.
Using substitution of variables, we can easily see that the number of weak
$m$-compositions of $n$ is equal to the number of $m$-compositions of $n
+ m$, which is $\binom{n+m-1}{m-1} = \binom{n+m-1}{n}$.

Problems on or related to compositions have been studied over algebraic
structures beyond the integers.
Let $\fq$ be a finite field of $q$ elements and let $\fq^*=\fq\setminus \{0\}$.
Li and Wan \cite{LiWan, LiWan:12} estimated the number
$$N(m, s, S) = \# \{ \{x_1, x_2, \ldots, x_m\}
\subseteq  S \mid  x_1 + x_2  + \cdots + x_m  = s \}$$
of $m$-subsets of $S
\subseteq \fq$ whose sum is $s \in \fq$.
In particular, exact formulas are obtained in cases where $S= \fq$ or $\fq^*$
or $\fq \setminus \{0, 1\}$. A shorter proof was given by Kosters
\cite{Kosters:13}  using character theory.
Motivated by the study by Li and Wan and the study of polynomials of prescribed
ranges over finite fields \cite{MuWa,gacsetal}, Muratovi\'{c} and Wang proposed to
study
$$c(m, s, S) = \# \{ (x_1, x_2, \ldots, x_m) \in S \times S \times \cdots
\times S \mid  x_1 + x_2  + \cdots + x_m = s \},$$
that is, the number of
ordered $m$-tuples  whose sum is $s$ and where each coordinate belongs to $S
\subseteq \fq$, as well as the number $M(m, s, S)$ which counts the number of
$m$-multisets over $S \subseteq \fq$ whose sum is $s \in \fq$.
In particular, when $S = \fq^*$,
this essentially gives the concept of
compositions and partitions over a finite field, respectively.

A \emph{partition} of $s\in \fq $ into $m$ parts is a multiset of $m$ elements
of $\fq^*$ whose sum is $s$.
The $m$ nonzero elements are refered to as the parts of the partition.
In \cite{MurWan2}, Muratovi\'{c} and Wang  obtain an exact formula for $M(m, s,
\fq^*)$,  the number of partitions of an element $s$ into $m$ parts over finite
field $\fq$.

A \emph{composition} of $s\in \fq $ with $m$ parts is a solution
$(x_1, x_2, \ldots, x_m)$  to the equation
\begin{equation} \label{eqCompo}
s=x_1+x_2+\cdots +x_m,
\end{equation}
with each $x_i \in \fq^*$. Similarly, a \emph{weak composition} of $w\in \fq $
with $m$ parts is
a solution $(x_1, x_2, \ldots, x_m)$ to Equation~(\ref{eqCompo}) with each $x_i \in \fq$.
We denote the number of compositions of $s$ having $m$ parts by
$c(m, s, \fq^*)$.
The number of weak compositions of $s$ with $m$ parts is denoted by $c(m, s,
\fq)$.
A formula for $c(m,s,\mathbb{F}_p)$ can be found in \cite[p.~295]{BEW}.
A general formula for $c(m, s, \fq^*)$ for arbitrary $q$  and nonzero $s$ can
be obtained using a remark on the normalized Jacobi sum of the trivial
character given in \cite{Conrad} (see Remark~1 on page 144).
A recurrence relation for $c(m,s,\fq)$ is given in \cite{MurWan}.

Counting compositions over finite fields where parts are restricted to
a subset is conceptually related to Waring's problem and solutions to diagonal
equations.
In number theory, Waring's problem asks whether each natural number $k$ has an
associated positive integer $m$ such that every natural number is the sum of at
most $m$ natural numbers to the power of $k$.
The problem was originally posed in 1770 and answered in the affirmative for
integers by Hilbert in 1909.
Since then, there has been a good deal of research on estimating the Waring's
number $g(k)$ for every $k$, which denotes the minimum number $m$ of $k$th
powers of naturals needed to represent all positive integers.

Over finite prime fields $\fp$, Waring's number $g(k, p)$ is  the smallest
number $m$ such that for all $a \in \fp$, the equation
\begin{equation} \label{eq:waring}
x_1^k + x_2^k + \cdots + x_m^k \equiv a \pmod{p}
\end{equation}
has a solution in integers $x_i$ (see for example~\cite{Cochrane}).
Waring's problem over finite fields is to estimate $g(k, p)$ and, when
possible, evaluate it.
This is equivalent to finding the smallest number $m$ such that
\begin{equation} \label{eq:waring}
x_1 + x_2 + \cdots + x_m \equiv a \pmod{p}
\end{equation}
has a solution such that $x_i \in S$ where $S$ denotes the subset of $\fp$
consisting of all $k$\textsuperscript{th} powers in $\fp$.
Essentially we need to find the number of solutions $(x_1, x_2, \ldots, x_m)$
in $S\times S \times \ldots \times S$ and pick the minimum $m$ so that this
number is always positive for any $a\in \fp$.
However, this is a computationally difficult task in general.

There is also extensive study on the number of solutions to \enquote{diagonal}
equations of the type
\begin{equation}\label{diagonal}
a_1 x_1^{d_1} + \cdots + a_m x_m^{d_m} = a,
\end{equation}
where $a_1, \ldots, a_m\in \fp^*$, $a\in \fp$ and $d_1, \ldots, d_n$ are
positive integers.
The pioneering work has been done by Weil \cite{Weil49}, who
expressed the number of solutions in terms of Jacobi sums.
Explicit
formulas for the number of solutions for certain choices of $a_1, \ldots, a_m,
a$, and exponents $d_1, \ldots, d_m$ can be deduced from Weil's expression;
one may consult
\cite{Baoulina10, Baoulina16,CaoSun07, Sun97, SunYuan96,Wolfmann92} and the
references therein for details.
For diagonal equations, the problem can be again viewed as a composition
problem over $\fp$, such that each part is restricted to lie in a coset of a
multiplicative subgroup.

Using the fact that the additive group $(\fq,+)$ is isomorphic to the
additive group $(\mathbb{F}_p^r,+)$,
we obtain that the numbers of partitions and compositions of
elements over $\mathbb{F}_p^r$ are the same as the numbers of partitions
and compositions of corresponding elements over $\fq$.
More generally, the number of partitions over an arbitrary finite abelian group
is given in \cite{MurWan2}.
However, problems of enumeration of compositions over an arbitrary finite
abelian group are largely open.
We note that there have been some studies on compositions over integer tuples
\cite{Lou08,MPR09}, which are also called matrix compositions.

In the present work, we consider two general problems on restricted
compositions over finite abelian groups.
Let ${\mZ}_k$ denote the additive cyclic  group $\{0,1,2,\ldots, k-1\}$ and
${\mZ}^{*}_k:=\{1,2,\ldots,k-1\}$.
Since a finite abelian group $G$ is isomorphic to a direct sum $\bigoplus_{t=1}^r
\mZ_{k_t}$, in the following we use $G$ to denote such a direct sum.
Sometimes it is convenient to add (tuples of) elements of $\mZ_k$ as integers,
so when performing group addition in $G$ we explicitly speak of modular
addition.
In the rest of the paper, we also assume $|G|\ge 2$, and consequently $k_t\ge
2$ for each $1\le t\le r$.
We also use $\bf 0$ to denote the zero element of $G$ and
adopt the notation $G^*:=G\setminus \{{\bf 0}\}$.

\begin{definition}[Compositions over a finite abelian group]
An \emph{$m$-composition of $\bs\in G$ over $G$} is a solution to
$\bx_1+\cdots +\bx_m = {\bs}$, where each $\bx_i=(x_{i, 1}, \ldots, x_{i, r})$
is a nonzero element of $G$ and addition is taken component-wise and modulo
$k_t$ for $1\le t\le r$.
If we allow ${\bf 0}$ as a value for the $\bx_i$, then we speak of \emph{weak
compositions} over $G$.
\end{definition}

Firstly we are interested in finding the number of $m$-compositions over $G$
such that for $1 \leq j \leq m$, part $\bx_j$ is restricted to an abitrary
subset $S_j \subseteq G$.
Our contribution is Theorem~\ref{thm:Subset},
where we obtain an asymptotic formula for the number of such $m$-compositions
as $m\to \infty$.

Secondly, we enumerate $m$-compositions over a finite abelian group $G$ such
that any $d$ consecutive parts satisfy certain conditions for a given positive
integer $d$.
These are called locally restricted compositions.
Bender and Canfield \cite{BC09} studied integer compositions under general
local restrictions.
For example, Carlitz compositions are those in which adjacent parts are
distinct.
In a private communication, G.\ L.\ Mullen   \cite{mullen} suggested the following problem:
Let $p$ be a prime, and let $d$ and $m$ be positive integers with $d \leq m$.
Let $N_m^{(d)}$ denote the number of solutions to the congruence $x_1 + \cdots
+ x_m \equiv 1 \pmod{p}$ where each subsum $\sum  x_j$ of at most $d$ parts is nonzero (modulo $p$).
Find $N_m^{(d)}$ when $d\geq 2$. Motivated by Mullen's problem, we consider a
related problem which falls within the locally-restricted framework.  Namely,
we estimate the number of solutions to the congruence $x_1 + \cdots + x_m
\equiv 1 \pmod{p}$ where  each subsum $\sum_{j=i}^{i+d-1}  x_j$  ($1\leq i \leq
m-d+1$) of at most $d$ consecutive parts is nonzero.  We call these
compositions {\it locally $d$-Mullen compositions}.
Generating functions for general locally restricted compositions over finite
abelian groups are given in Proposition~\ref{prop:local}.
Under moderate conditions, the asymptotic number of locally restricted
compositions is given in Theorem~\ref{thm:local}. As consequences, we obtain
asymptotics for a few concrete composition problems including locally
$d$-Mullen compositions (see Corollary~\ref{cor:local}, 
Corollary~\ref{cor:list},  and Theorem~\ref{theorem:locallyMullen}).

We present these main results in Section~\ref{sec:results}.
Some auxiliary propositions are discussed in Section~\ref{sec:propositions} and
the proofs of the main results are given in Section~\ref{sec:proofs}.
In Section~\ref{sec:biject} we present bijections which help illuminate
locally $d$-Mullen compositions.
Finally we give conclusions in Section~\ref{sec:conclusion}.

\section{Main results} \label{sec:results}
Recall that $G$ denotes the set of integer tuples in the direct sum
$\bigoplus_{t=1}^r\mZ_{k_t}$.
In the following we denote the order of $G$ by $|G|:=\prod_{t=1}^rk_t$.

\begin{theorem}[Subset restriction]\label{thm:Subset}
For each $j \in \{1, 2, \ldots\}$, let $S_j$ be the Cartesian product
$\prod_{t=1}^rS_{t,j}$ where
$S_{t,j}$ is a given subset of $\mZ_{k_t}$.
Let $c(\bs;S_1,\ldots,S_m)$ be the
number of  $m$-compositions of $\bs$ over
$G$ such that the $j\textsuperscript{th}$ part lies in
$S_j$, for each $j \in \{1, \ldots, m\}$.
Assume that for $1\le t\le r$ and $j \geq 1$ we have
$$\gcd\{a-b:a,b\in S_{t,j}\}=1.$$
Then, as $m \to \infty$,
$$c(\bs;S_1,\ldots,S_m)=\frac{1}{|G|}\left(\prod_{j=1}^m|S_{j}|\right)
  (1+O(\theta^m)),$$
where $0<\theta<1$ is a constant independent of $m$.
\end{theorem}

Let \(\SEQ_{m}(G) := \left\{ (\bx_1, \bx_2, \ldots, \bx_m): \bx_j \in G
\right\} = G^m\) denote the class of all possible weak $m$-compositions over
$G$.
The \emph{size} of a composition, denoted $|(\bx_1, \bx_2, \ldots, \bx_m)|$, is
defined to be the sum $\sum_{j=1}^m \bx_j$, using component-wise integer
addition.
Let $\SEQ(G) = \cup_{m\geq 0} \SEQ_{m}(G)$.
A subclass $\mathcal{A} \subseteq \SEQ(G)$ is called \emph{locally restricted}
if any $d$ consecutive parts satisfy certain restrictions for a given positive
integer $d$.
General locally restricted integer compositions were first studied in
\cite{BC09} using local restriction functions.
They can also be defined using local restriction digraphs.
For the purpose of this paper, we use the following definition, which is
essentially from \cite{BG16}.

\begin{definition}[Locally restricted compositions associated with a digraph] \label{def3}
Let $\spn$ be a positive integer.
We use $\eps$ to denote the empty composition, and we use $\eps_s$ to denote
a distinguished copy of $\eps$.
Let $\sT=\SEQ_{0}(G) \cup \cdots \cup \SEQ_{\si-1}(G)$, and let
$\sR$ be a nonempty subset of $\SEQ_{\spn}(G)$.
Let $D$ be a digraph with vertex set $V(D)=\{\eps_s\} \disjun \sR \disjun \sT$.
Assume $D$ satisfies the following conditions.
\begin{enumerate}
\item There is at least one arc from $\eps_s$ to $\sR$ and at least one arc from
  $\sR$ to $\sT$.
\item There are no arcs (including loops) between vertices in
  $\{\eps_s\} \disjun \sT$.
\item The sub-digraph $D_{\sR}$ of $D$ induced by $\sR$ is strongly
  connected and contains at least two vertices.
\end{enumerate}
The vertices in $\sR$ are called \emph{recurrent} vertices.
Let $\sW$
denote the set of directed walks from $\eps_s$ to $\sT$, and let $\SEQ(G;D)$
denote the class of compositions obtained by concatenating vertices in $\sW$.
Then $\SEQ(G;D)$ is called the class of \emph{locally restricted compositions
over $G$ associated with $D$}.
We call $\spn$ the {\em span} of $\SEQ(G;D)$.
\end{definition}

Informally, we are defining $\SEQ(G;D)$ by specifying two sets $\sR$ and $\sT$
of ``building block'' compositions, and saying how they may be combined
sequentially to form the compositions of interest:
Two blocks are allowed to be joined if there is an arc between them in $D$.
The structure of the graph implies that sequences of vertices forming a walk
from $\eps_s$ to $\sT$ are made up of one or more elements of $\sR$ followed
by one element of $\sT$.
Defining combinatorial objects by walks in graphs is a standard technique; see
e.g.\ \cite[Sec.~4.7]{Stanley}.

\begin{example}[name=Nonzero adjacent sum over $\Zk$, label=exa:nonzero]
Take compositions over $\Zk$ such that the sum of two adjacent parts is
nonzero, modulo $k$.
We can represent these compositions using a local restriction digraph as
follows.
In the digraph $D$, set $\sR=\mathbb{Z}_k^* \subset \SEQ_{1}(\Zk)$, and
$\sT=\SEQ_0(\Zk)=\{\eps\}$.
Include an arc from $u$ to $v$ in $\sR$ if and only if $u+v\not\equiv 0$ (mod $k$).
In this case, we have locally restricted compositions over $\Zk$  with span
$\sigma=1$.
The digraph for weak compositions is defined similarly with $\sR=\Zk =
\SEQ_{1}(\Zk)$.
\end{example}

\begin{example}[name=Carlitz compositions over $\Zk$, label=exa:carlweak]
Recall that Carlitz compositions are those where adjacent parts are distinct.
A corresponding local restriction digraph $D$ has vertex sets
$\sR=\mathbb{Z}_k^* \subset \SEQ_{1}(\Zk)$
(or $\sR=\Zk = \SEQ_{1}(\Zk)$ for weak compositions), $\sT=\{\eps\}$, and there
is an arc from any vertex in $\sR$ to any different vertex in $\sR$. In this
case, we have  locally restricted compositions over $\Zk$  with span
$\sigma=1$.
\end{example}

We note that it is possible for two different digraphs $D_1$ and $D_2$ to
define equal families $\SEQ(G; D_1)$ and $\SEQ(G; D_2)$ of compositions over
$G$.
For example, Carlitz compositions can also be defined using a digraph with span
$\spn=2$ such that each recurrent vertex is a pair of distinct nonzero elements
of $G$.

\begin{example}[name=Locally Mullen compositions, label=exa:Mullen]
Recall that locally $d$-Mullen compositions over $G$ are those such that the
sum of at most $d$ consecutive terms is nonzero in $G$.
Example~\ref{exa:nonzero} gives locally 2-Mullen compositions.
A digraph $D$ with span $3$ for locally 3-Mullen compositions can be
defined as follows.
Let $\sR$ be the set of locally 3-Mullen compositions with 3 parts over $G$.
We join $\eps_s$ to every vertex in $\sR$, and
join a vertex $\bu\in \sR$ to a vertex $\bv\in \sR\cup \sT$ if the
concatenation $\bu\bv$ is a locally 3-Mullen composition.
\end{example}

\begin{theorem}[Local restriction]\label{thm:local}
Consider a class $\SEQ(G; D)$ of locally restricted compositions with span
$\spn$. Suppose the following are satisfied.
  \begin{enumerate}

    \item The greatest common divisor of the lengths of all directed cycles in
      $D_{\sR}$ is equal to $1$.
    \item For each $1\le t\le r$, there is a positive integer $\ell$,
      two recurrent vertices $\bu$ and $\bv$, and a nonempty set
      $\sW_{\sR}$ of directed walks in $D_{\sR}$ of length $\ell$ from $\bu$ to
      $\bv$ such that the following hold.
      \begin{enumerate}
        \item For each walk $W=\bu\bu_1\ldots\bu_{\ell-1}\bv$ in $\sW_{\sR}$
          and each $i\ne t$,
          we have $u_{i,j}=\phi_{i,j}$,
          where $u_{i,j}$ denotes  the $i\textsuperscript{th}$ coordinate of
          the size vector $|\bu_j|$   and $\phi_{i,j}$ is fixed over all walks
          in  $\sW_{\sR}$.
        \item We have $\gcd\{m-n:m,n\in N\}=1$, where $N=\{n:n=u_{t,1}+\cdots+
          u_{t,\ell-1} \hbox{ for some walk
          $\bu\bu_1\ldots\bu_{\ell-1}\bv\in \sW_{\sR}$} \}$.
      \end{enumerate}
  \end{enumerate}
  Let $m=a\spn +b$ for some integers $a,b$ with $a\ge 1$ and $0\le b<\si$, and let $c_m(\bs,D)$ be the
  number of $m$-compositions of
  $\bs$ in $\SEQ(G; D)$.
  Fix $b$ but allow $a$ to vary.
  Suppose further that there is at least one arc from $\sR$ to $\SEQ_b(G)$.
  Then there are constants $A>0, B>1$ and $0<\theta<1$ (that is, independent of
  $m$ but perhaps depending on $b$) such that
  $$c_m(\bs,D)=A\cdot B^{m}\left(1+O\left(\theta^{m}\right)\right),
  \qquad m=a\spn+b,~m\to\infty.$$
\end{theorem}

\noindent {\bf Remark.} Assumption~2 in Theorem~2 is a technical condition
which allows our proofs to work.
It might be possible to relax this technical assumption.

\begin{corollary}\label{cor:local}
Assume the notation and conditions of Theorem~2. Let $H$ be the outdegree of $\eps_s$.
Suppose further that there are constants  $J$ and $K$ such that
\begin{enumerate}
  \item every recurrent vertex has outdegree $K$ in $D_{\sR}$, and
  \item every recurrent vertex is joined to $J$ vertices in $\SEQ_b(G)$.
\end{enumerate}
  Then Theorem~\ref{thm:local} holds with
\begin{equation}\label{cor1}
A=\frac{H\,J}{|G|K^{1+b/\spn}},~B=K^{1/\spn}.
\end{equation}
\end{corollary}

In the following we adopt the notation
$\displaystyle x^{\underline k}=x(x-1)\cdots (x-k+1).$

\begin{corollary}\label{cor:list}
Let $c_m (\bs)$ be the number of $m$-compositions of $\bs$ for each of the
following classes of compositions over $G$.
Then
  $$c_m(\bs)=A\cdot B^{m}\left(1+O\left(\theta^{m}\right)\right),
  \qquad m\to\infty,$$
where $A,B$ are given below.
\begin{enumerate}
  \item  For weak compositions such that there is no repeated part
    among any $d+1$ consecutive parts and $|G|\ge d+2$,  $$A=\frac{1}{|G|}|G|^{\underline d}(|G|-d)^{-d} ,~
   B=|G|-d.$$
  \item For compositions  such that there is no repeated part among
    any $d+1$ consecutive parts and $|G|\ge d+3$,  $$A=\frac{1}{|G|}(|G|-1)^{\underline d}(|G|-1-d)^{-d},~
   B=|G|-1-d.$$
  \item For weak compositions such that there is no repeated part
    among any $d+1$ consecutive parts and the first $d$ parts are nonzero, and
    $|G|\ge d+2$,
  $$A= \frac{1}{|G|}(|G|-1)^{\underline d}(|G|-d)^{-d},~B=|G|-d.$$
\item For weak compositions such that the sum of any $d$
    consecutive parts is nonzero and $|G|\ge 3, d\ge 2$,
   $$A=|G|^{d-2}(|G|-1)^{1-d},~
   B=|G|-1.$$
\item For compositions over $\fq$ such that the product of any $d$
    consecutive parts is not equal to $\bf 1$ and $q\ge 4, d\ge 2$,
   $$A=\frac{1}{q}(q-1)^{d-1}(q-2)^{1-d},~
   B=q-2.$$
\end{enumerate}
\end{corollary}

Next we give an example to which Corollary~\ref{cor:local} does not apply.

\begin{example}[label=exa:d3]
Consider the class of compositions over $G$ such that the sum of any three
consecutive parts is nonzero.
A corresponding restriction digraph $D$ is defined as follows.
The set $\sR$ consists of all ordered triples of nonzero elements of $G$
whose sum is nonzero.
We note that the sum of two consecutive parts might be
zero.
The vertex $\eps_s$ is joined to every vertex in $\sR$, and every
recurrent vertex is joined to the vertex $\eps$ in $\sT$.
A recurrent vertex
$\bu$ is joined to a vertex $\bv\in \sR\cup \SEQ_1(G^*) \cup \SEQ_2(G^*)$ if
the sum of any three consecutive parts in the concatenation $\bu\bv$ is
nonzero.
To compute the outdegree of each recurrent vertex, we need to distinguish two
cases.\\
Case~1: the recurrent vertex $\bu:=(u_1,u_2,u_3)$ where $u_1,u_2,u_3\ne {\bf
0}$ and $u_2+u_3={\bf 0}$.
Such a vertex is joined to a recurrent vertex
$\bv:=(v_1,v_2,v_3)$ if $v_1,v_2,v_3\ne {\bf 0}$, $u_3+v_1+v_2\ne {\bf 0}$ and
$v_1+v_2+v_3\ne {\bf 0}$.
The outdegree of $\bu$ in $\sR$ is equal to
\begin{align*}
&~|\{(v_1,-v_1,v_3):v_1,v_3\ne {\bf 0}\}|\\
&+|\{(v_1,v_2,v_3):v_1\ne {\bf 0};v_2\ne -v_1,-v_1-u_3,{\bf 0}; v_3\ne -v_1-v_2,{\bf 0} \}|\\
=& ~(|G|-1)^2+(|G|-1)(|G|-3)(|G|-2).
\end{align*}
The outdegree of $\bu$ in $\SEQ_1(G)$ is clearly equal to $|G|-1$.
The outdegree of $\bu$ in $\SEQ_2(G)$ is equal to
\begin{align*}
& |\{(-u_3,v_2):v_2\ne {\bf 0}\}|+|\{(v_1,v_2):v_1\ne -u_3,{\bf0};v_2\ne -v_1-u_3,{\bf 0},\}|\\
=& (|G|-1)+(|G|-2)^2.
\end{align*}
\noindent Case~2: the recurrent vertex $\bu:=(u_1,u_2,u_3)$ where
$u_1,u_2,u_3\ne {\bf 0}$ and $u_2+u_3\ne {\bf 0}$.
The outdegree of $\bu$ in $\sR$ is equal to
\begin{align*}
&~|\{(v_1,-v_1,v_3):v_1\ne -u_2-u_3,{\bf 0};v_3\ne {\bf 0}\}|\\
&+|\{(v_1,v_2,v_3):v_1\ne -u_2-u_3,{\bf 0};v_2\ne -v_1, -v_1-u_3,{\bf 0};v_3\ne -v_1-v_2,{\bf 0} \}|\\
=&~(|G|-2)(|G|-1)+(|G|-2)(|G|-3)(|G|-2).
\end{align*}
The outdegree of $\bu$ in $\SEQ_1(G)$ is clearly equal to $|G|-2$.
The outdegree of $\bu$ in $\SEQ_2(G)$ is equal to
\begin{align*}
& |\{(-u_3,v_2):v_2\ne {\bf 0}\}|+|\{(v_1,v_2):v_1\ne -u_3,-u_2-u_3,{\bf 0};v_2\ne -u_3-v_1,{\bf 0},\}|\\
=& (|G|-1)+(|G|-3)(|G|-2).
\end{align*}

We show later that in fact Theorem~\ref{thm:local} can still apply to this
class.
\end{example}

\section{Propositions} \label{sec:propositions}

In this section we present results which are used to prove our main
theorems.

The following multivariate ``multisection formula'' might be known (the univariate case can be found in \cite[Ex.\ 1.1.9]{gold}), but we are unable to find a reference.
So we also include a proof.

For two vectors $\bz=(z_1,z_2,\ldots,z_r)$ and $\bn=(n_1,n_2,\ldots,n_r)$,  we
use $\bz^{\bn}$ to denote the product $\prod_{j=1}^rz_j^{n_j}$. In the rest of
the paper, we use the Iverson bracket $[P]$ which is equal to 1 if the
statement $P$ is true and 0 otherwise.
We also use $\om_k:=\exp(2\pi i/k)$ to denote the $k$\textsuperscript{th}
primitive root of unity which has the property
$$\sum_{j=0}^{k-1}\om_k^{sj}=k\left[s\equiv 0 \pmod k\right].$$

\begin{proposition}[Multivariate multisection formula]\label{prop:multisection}
Let $A(\bz)=\sum_{\bn}a_{\bn}\bz^{\bn}$ be a multivariate generating function,
where the indeterminate is $\bz:=(z_1,\ldots,z_r)$ and the sum is over all
$\bn:=(n_1,\ldots,n_r)$.
For any $\bs=(s_1,\ldots, s_r) \in G$, we have
$$\sum_{\bn\equiv\bs \pmod{\mathbf{k}}}
  a_{\bn}\bz^{\bn}=\left(\prod_{t=1}^r\frac{1}{k_t}\right)\sum_{\bj\in
  G}\left(\prod_{t=1}^r \om_{k_t}^{-j_ts_t}
  \right)A\left(z_1\om_{k_1}^{j_1},\ldots,z_r\om_{k_r}^{j_r}\right),$$
where $\om_{k_t} =\exp(2\pi i/k_t)$ is a  primitive $k_t$\textsuperscript{th}
root of unity for $1\leq t \leq r$.
\end{proposition}
\begin{proof}
 We have
\begin{align*}
&  \sum_{\bj\in
  G}\left(\prod_{t=1}^r \om_{k_t}^{-j_ts_t}
  \right)A\left(z_1\om_{k_1}^{j_1},\ldots,z_r\om_{k_r}^{j_r}\right)\\
  =&\sum_{\bn}a_{\bn}\bz^{\bn}\prod_{t=1}^{r}\sum_{j_t=0}^{k_t-1}
    \om_{k_t}^{j_t(n_t-s_t)} \\
  =&\sum_{\bn}a_{\bn}\bz^{\bn}\prod_{t=1}^{r}k_t\left[n_t\equiv s_t
    \pmod{k_t}\right]\\
  =&\prod_{t=1}^{r}k_t\sum_{\bn}a_{\bn}\bz^{\bn}\left[\bn\equiv \bs
    \pmod{\mathbf{k}}\right]. \qedhere
\end{align*}
\end{proof}

\begin{corollary} \label{cor:integer}
In this corollary compositions may be weak.
Fix some class $\mathcal{A} \subseteq \SEQ(G)$ of restricted compositions over
$G$.
For $\bs \in G$, let
  \[c_m(\bs) = \left|\{\mathbf{x}: \mathbf{x} = (\bx_1, \bx_2, \ldots, \bx_m)
  \in\mathcal{A}, |\mathbf{x}| \equiv \bs
  \pmod{\mathbf{k}}\}\right|.\]
For $\bn \in \mathbb{Z}_{\geq 0}^r$, we define the \emph{integer composition}
counting sequence
\[\hat{c}_m(\bn) = \left|\{\mathbf{x}: \mathbf{x} = (\bx_1, \bx_2, \ldots,
  \bx_m)\in\mathcal{A}, |\mathbf{x}| = \bn\}\right|,\]
and the generating function
\( \hat{C}_m(\bz) = \sum_{\bn} \hat{c}_m(\bn)\bz^\bn. \)
Then
\[
  c_m(\bs) =
\left(\prod_{t=1}^r\frac{1}{k_t}\right)\sum_{\bj\in
  G}\left(\prod_{t=1}^r \om_{k_t}^{-j_ts_t}
  \right)
  \hat{C}_m\left(\om_{k_1}^{j_1},\ldots,\om_{k_r}^{j_r}\right).
\]
\end{corollary}
\begin{proof}
Immediate from Proposition~\ref{prop:multisection}.
\end{proof}

\begin{example}[name=Unrestricted $m$-compositions over $G$]
The corresponding generating function for integer $m$-compositions
is
\[\hat{C}_m(\bz)=\left(\sum_{\bj\in
  G}\bz^{\bj}-1\right)^m=\left(\prod_{t=1}^r\sum_{j_t=0}^{k_t-1}z_t^{j_t}-1\right)^m
  = \left(\prod_{t=1}^r\frac{1-z_t^{k_t}}{1-z_t}-1\right)^m.
\]
It follows from Corollary~\ref{cor:integer} that the number of $m$-compositions
of $\bs$ over $G$ is
\begin{align*}
  c_m(\bs)
  &= \left(\prod_{t=1}^r\frac{1}{k_t}\right)\sum_{\bj\in
    G}\left(\prod_{t=1}^r\om_{k_t}^{-j_ts_t}\right)\hat{C}_m(\om_{k_1}^{j_1},\ldots,\om_{k_r}^{j_r})\\
  &= \left(\prod_{t=1}^r\frac{1}{k_t}\right)\sum_{\bj\in
    G}\left(\prod_{t=1}^r\om_{k_t}^{-j_ts_t}\right)\left(\left[\bj={\bf
    0}\right]\prod_{t=1}^rk_t-1\right)^m\\
  &= \frac{1}{|G|}\left(\left(|G|-1\right)^m+(-1)^m
    \sum_{\bj\in G^*}\prod_{t=1}^r\om_{k_t}^{-j_ts_t}\right)\\
  &= \frac{1}{|G|}\left(\left(|G|-1\right)^m+(-1)^m
    \left(\prod_{t=1}^r\sum_{j_t=0}^{k_t-1}\om_{k_t}^{-j_ts_t}-1\right)\right)\\
  &= \frac{1}{|G|}\left(\left(|G|-1\right)^m+(-1)^m\left(\left[\bs={\bf
    0}\right]|G|-1\right)\right).
\end{align*}
Thus
\begin{align*}
  c_m(\bs) &= \frac{1}{|G|}\left(\left(|G|-1\right)^m-(-1)^m\right),~\hbox{ if }\bs\ne {\bf 0},\\
  c_m({\bf 0}) &=\frac{1}{|G|}
  \left(\left(|G|-1\right)^m+(-1)^m\left(|G|-1\right)\right).
\end{align*}
If $r=1$, this is the result given in \cite{MurWan}.
  We also note that $$c_m(\bs)\sim \frac{1}{|G|}\left(|G|-1\right)^m$$ for each
  $\bs \in G$  as $m\to \infty$, which agrees with the result given by
  Theorem~1.
\end{example}

\begin{example}[name=Unrestricted weak $m$-compositions over $G$]
The corresponding generating function for integer compositions is
$$\hat{C}_m(\bz)=\left(\prod_{t=1}^r\frac{1-z_t^{k_t}}{1-z_t}\right)^m.
$$
Applying Corollary~\ref{cor:integer}, we have
\begin{align*}
c_m(\bs) &= \left(\prod_{t=1}^r\frac{1}{k_t}\right)\sum_{\bj\in
    G}\left(\prod_{t=1}^r\om_{k_t}^{-j_ts_t}\right)\left(\prod_{t=1}^r\frac{1-\om_{k_t}^{j_tk_t}}{1-\om_{k_t}^{j_t}}\right)^m\\
  &= \left(\prod_{t=1}^r\frac{1}{k_t}\right)\sum_{\bj\in
    G}\left(\prod_{t=1}^r\om_{k_t}^{-j_ts_t}\right)\left(\left[\bj={\bf
    0}\right]\prod_{t=1}^rk_t\right)^m\\
  &= \left(\prod_{t=1}^r\frac{1}{k_t}\right)\left(\prod_{t=1}^rk_t\right)^m\\
  &= \left(\prod_{t=1}^rk_t\right)^{m-1}.
\end{align*}
This result can also be obtained by a direct counting argument: the first $m-1$
parts can be constructed in $\left(\prod_{t=1}^rk_t\right)^{m-1}$ ways, and the
last part is then uniquely determined by the equation
$\bx_1+\cdots+\bx_{m-1}+\bx_m = \bs$
in $G$.
\end{example}

\begin{example}
Let $c_m$ be the number of solutions to Equation (\ref{diagonal}). Let $y_j=a_jx_j^{d_j}$ and $S_j=\{a_jx_j^{d_j}:x_j\in \fp^*\}$.  Then $c_m$ is the number of $m$-compositions $y_1,\ldots,y_m$ of $a$
over $\fp$ such that $y_j\in S_j$. We have
$$|S_j|=|\{a_jx_j^{d_j}:x_j\in \fp\}|=|\{x_j^{d_j}:x_j\in \fp\}|=\frac{p-1}{\gcd(d_j, p-1)}.$$
  Viewing elements of $S_j, S_j-S_j$ as integers less than $p$,   the condition
\[\gcd \{ a_j s^{d_j}-a_j t^{d_j}: s,t\in \fp^*\}=1, \]
is satisfied when   $S_j - S_j$ contains two relatively prime integers.
\end{example}

\begin{definition}[Transfer matrix, weights, start and finish vectors]
  \label{def:transfer}
Let $\SEQ(G; D)$ be a class of locally restricted compositions from
Definition~\ref{def3}.
Fix an order on the vertices in $\sR$.
We define the \emph{transfer matrix} $T(\bz)$ as the square matrix whose rows
and columns are indexed by the vertices in $\sR$ such that
$T_{\bu,\bv}(\bz)=\bz^{|\bv|}$ if there is an arc from
$\bu$ to $\bv$, and $T_{\bu,\bv}(\bz)=0$  otherwise.
We say that $\bz^{|\bv|}$ is the \emph{weight} of an arc $(\bu,\bv)$.
The weight of a directed walk in $D$ is defined to be the product of the
weights of the arcs of the walk.
The \emph{start vector} $\start(\bz)$ is defined as the row vector whose
$j\textsuperscript{th}$ entry is the weight of the arc from $\eps_s$ to the
$j\textsuperscript{th}$ recurrent vertex (this entry is zero if such an arc
does not exist).
For each integer $0\le b<\si$, we define the \emph{finish vector}
$\term_b(\bz)$ as the column vector whose $j\textsuperscript{th}$ entry is the
sum of weights of all arcs from the $j\textsuperscript{th}$ recurrent vertex
to a vertex in $\SEQ_b(G)$.
\end{definition}

\begin{example}[name=Nonzero adjacent sum, continues=exa:nonzero]
For weak compositions, the transfer matrix $T(z)$ has
size $k\times k$ where $T_{i,j}(z)=z^j$ if $i+j\not\equiv 0 \pmod{k}$ and
$T_{i,j}(z)=0$ if $i+j \equiv 0 \pmod{k}$.
The start and finish vectors are, respectively,
$\start(z)=\rvect{1,z,\cdots,z^{k-1}}$ and $\term_0(z)=\rvect{1,1,\cdots,1}^\top$.
For example, if $k=3$, we have
\[
T(z)=\left[\begin{array}{cccc}
0&z&z^2\\
1&z&0\\
1&0&z^2
\end{array}\right],~\start(z)=\rvect{1,z,z^2},~\term_0(z)=\rvect{1,1,1}^\top.
\qedhere
\]
\end{example}

Now we are ready to use the transfer matrix to enumerate locally restricted
compositions in $\SEQ(G;D)$.
\begin{proposition}\label{prop:local}
Let $m, \spn$ be positive integers and define integers $a,b$ such that
$m=a\spn +b$ where $0\leq b < \spn$.
Let $\bs=(s_1,\ldots, s_r)$ be a member of $G$ ($=\sum_{j=1}^r \mZ_{k_j}$),
and let $\SEQ(G;D)$ be a class of locally restricted compositions with span
$\spn$.
As in Theorem~\ref{thm:local}, we let $c_m(\bs,D)$ denote the number of
$m$-compositions of $\bs$ in $\SEQ(G;D)$.
Let $T(\bz), \start(\bz), \term_b(\bz)$ be the corresponding transfer matrix,
start vector, and finish vector, and define
$$\hat{C}_m(\bz)=\start(\bz)T^{a-1}(\bz)\term_b(\bz).$$
Then we have
$$c_m(\bs,D)=\left(\prod_{t=1}^r\frac{1}{k_t}\right)\sum_{\bj\in
  G}\left(\prod_{t=1}^r \om_{k_t}^{-j_ts_t}
  \right)\hat{C}_m\left(\om_{k_1}^{j_1},\ldots,\om_{k_r}^{j_r}\right).$$
\end{proposition}
\begin{proof}
We note that the $(\bu,\bv)$ entry in $T(\bz)^{a-1}$ is the sum of weights of
all directed walks of length $a-1$ (that is, containing $a$ vertices) in $D_{\sR}$ from vertex $\bu$ to vertex $\bv$.
Since each vertex in $D_{\sR}$ is a sequence of parts from $G$ of length
$\spn$, it follows from the definition of $\start$ and $\term_b$ that
$\hat{C}_m(\bz)$ is the generating function of all directed walks
from $\eps_s$ to $\sT$ totaling $m=a\spn +b$ parts.
The result follows by applying Corollary~\ref{cor:integer}.
\end{proof}

\begin{example}[name=Carlitz compositions over $\Zk$, continues=exa:carlweak]
For Carlitz weak  compositions, it is clear that every recurrent vertex has
in-degree and out-degree equal to $k-1$ in $D_\sR$.
The transfer matrix $T(z)$ is given by $T_{i,i}(z)=0$ and $T_{i,j}(z)=z^j$ if
$i \ne j$.
The start and finish vectors are, respectively,
$\start(z)=\rvect{1,z,\cdots,z^{k-1}}$ and $\term_0(z)=\rvect{1,1,\cdots,1}^\top$.
So the generating function $\hat{C}_m(z)$
from Proposition~\ref{prop:local} for Carlitz weak
compositions is given by
$$\hat{C}_m(z)=\rvect{1,z,\cdots,z^{k-1}}T^{m-1}(z)\rvect{1,1,\cdots, 1}^\top,~~m\ge
  1.$$
For Carlitz compositions,
the transfer matrix $T(z)$ is obtained from the one above by
deleting row 1 and column 1.
The start and finish vectors are, respectively,
$\start(z)=\rvect{z,z^2, \cdots,z^{k-1}}$ and $\term_0(z)=\rvect{1,1,\cdots,1}^\top$.
This gives
\[
\hat{C}_m(z)=\rvect{z, z^2, \cdots,z^{k-1}}T^{m-1}(z)\rvect{1,1,\cdots,1}^\top,~~m\ge 1.
\qedhere
\]
\end{example}

The following two propositions are used later to derive the asymptotic number
of compositions over $G$ from their generating functions.
Proposition~\ref{prop:aper} below shows that, under an aperiodicity condition,
a polynomial attains a unique maximum absolute value on the unit disc at $1$.
Its proof can be found in \cite[Lemma~1]{GaoWor99}.
Proposition~\ref{prop:eigen} is essentially the Perron-Frobenius theorem and its
proof can be found in \cite{BP}.

\begin{proposition}\label{prop:aper}
Let $F(z)=\sum_{j\ge 0}f_jz^j$ be a polynomial with nonnegative coefficients.
Define $J=\{j:f_j>0\}$.
Suppose $J$ is not empty and $\gcd \{j-k: j,k\in J\}=1$.  Then
$|F(z)|<F(1)$ for all $|z|\le 1,~z\ne 1$.
\end{proposition}

\begin{proposition}\label{prop:eigen}
Let $T(\bz)$ be a transfer matrix as in Definition~\ref{def:transfer},
and let $\rho(\bz)$ denote the spectral radius of $T(\bz)$.
Suppose the greatest common divisor of all directed cycle lengths of the
digraph $D_{\sR}$ is equal to $1$.
Then we have the following.
\begin{enumerate}
\item The value $\rho:=\rho({\bf1})$ is a simple eigenvalue of $T({\bf1})$ and
  the corresponding eigenspace is spanned by a positive vector.
\item All other eigenvalues of $T({\bf1})$ are smaller, in modulus, than
  $\rho$.
\item Suppose $|\bz|\ne 1, |z_t|\le 1$ for all $t$, and $T(\bz)\ne T({\bf 1})$.
  Then $\rho(\bz)<\rho$.
\end{enumerate}

\end{proposition}

\section{Proofs} \label{sec:proofs}

\begin{proof}[Proof of Theorem~\ref{thm:Subset}]
The generating function for integer compositions is
  $$\hat{C}_m(\bz)
  =\prod_{j=1}^m \sum_{\bn\in S_j}\bz^{\bn}
  =\prod_{j=1}^m\prod_{t=1}^r\sum_{n_t\in S_{t,j}}z_t^{n_t}.$$
We note $\hat{C}_m({\bf 1})=\prod_{j=1}^m|S_j|$.
It follows from Corollary~\ref{cor:integer} that
\begin{align*}
  c(\bs;S_1,\ldots,S_m) &= \left(\prod_{t=1}^r\frac{1}{k_t}\right)\sum_{{\bf g}\in
    G}\left(\prod_{t=1}^r\om_{k_t}^{-g_ts_t}\right)
    \hat{C}_m(\om_{k_1}^{g_1},\ldots,\om_{k_r}^{g_r}) \\
  &= \frac{1}{|G|}\left( \prod_{j=1}^m |S_j| + \sum_{{\bf g}\in
    G^*}\left(\prod_{t=1}^r\om_{k_t}^{-g_ts_t}\right)
    \hat{C}_m(\om_{k_1}^{g_1},\ldots,\om_{k_r}^{g_r})\right).
\end{align*}

For a subset $A$ of $\mZ_{k_t}$,  define $P(z;A)=\sum_{a\in A}z^a$.
Let $\mathcal{I}:=\{A\subseteq \mZ_{k_t}: |P(z;A)|<|A| \hbox{ when  $|z|=1$ and
$z\ne 1$}\}$ and define $$\theta=\max
\left\{\frac{|P(\om_{k_t}^{g_t};A)|}{|A|}:1\le g_t\le k_t-1,1\le t\le r, A\in
\mathcal{I}\right\}.$$
Then $0\le \theta<1$. It follows from Proposition~\ref{prop:aper} that
$S_{t,j}\in \mathcal{I}$ and consequently $|P(z;S_{t,j})|<\theta|S_{t,j}|$.
Let $\mathbf{g}\in G$ satisfy $g_{t^*}>0$ for some $1\le t^*\le r$.
It follows that
\begin{align*}
  |\hat{C}_m(\om_{k_1}^{g_1},\ldots,\om_{k_r}^{g_r})|
   &=\left(\prod_{j=1}^m \left|\sum_{a\in S_{t^*,j}}\om_{k_{t^*}}^{ag_{t*}}\right|\right)\left(\prod_{t=1,t\ne t^*}^r\prod_{j=1}^m \left|\sum_{n_t\in S_{t,j}}\om_{k_t}^{n_tg_t}\right|\right) \\
   &=\left(\prod_{j=1}^m \left|P_{t^*,j}(\om_{k_{t^*}}^{g_{t*}})\right|\right)\left(\prod_{t=1,t\ne t^*}^r\prod_{j=1}^m \left|P_{t,j}(\om_{k_t}^{g_t})\right|\right) \\
  &<\left(\prod_{j=1}^m \theta |S_{t^*,j}|\right)\left(\prod_{t=1,t\ne t^*}^r\prod_{j=1}^m |S_{t,j}|\right) \\
  &= \theta^{m} \prod_{t=1}^r \prod_{j=1}^m |S_{t,j}| \\
  &= \theta^{m} \prod_{t=1}^r |S_t|.
\end{align*}
Now Theorem~1 follows immediately.
\end{proof}

\begin{proof}[Proof of Theorem~\ref{thm:local}]
As in Proposition~\ref{prop:eigen}, let $T(\bz)$ be the transfer matrix
of $\SEQ(G; D)$
and define $\rho(\bz)$ to be the absolute value of the dominant eigenvalue of
$T(\bz)$.
We abbreviate $\rho(\mathbf{1})$ simply as $\rho$.
Define
\begin{align*}
\theta_1 &= \max \left\{ \frac{|\lambda|}{\rho}: \lambda \hbox{ is any
    eigenvalue of $T({\bf1})$ other than } \rho \right\},\\
\theta_2 &= \max  \left\{\frac{\rho(\om_{k_1}^{j_1},\ldots,\om_{k_r}^{j_r})}{\rho} : 0\le
    j_t\le k_t-1,\, 1\le t\le r, {\bf j}\ne {\bf 0} \right\}.
\end{align*}
It follows from item~2 of Proposition~\ref{prop:eigen} that
$0\le\theta_1<1$.
 Recall that the $(\bu,\bv)$-entry of $T^{\ell}(\bz)$, denoted by
 $\left(T^{\ell}(\bz)\right)_{\bu,\bv}$, is equal to the sum of the weights of
 all directed walks from vertex $\bu$ to vertex $\bv$ of length $\ell$.
 Let ${\bf j}=(j_1,\ldots,j_r)$ satisfy $0\le j_i\le k_i-1$ and $j_t\ne 0$
 for some $1\le t\le r$. Let $\sW_{\sR}$ be the set of directed walks given in
 condition~2.
Dividing the set of all directed walks of length $\ell$ from $\bu$ to $\bv$
into $\sW_{\sR}$ and its complement ${\bar \sW_{\sR}}$,  we obtain
\begin{align*}
  &\left|\left(T^{\ell}(\om_{k_1}^{j_1},\ldots,\om_{k_r}^{j_r})
    \right)_{\bu,\bv}\right|\\
  \le &  \left|\sum_{W\in
    \sW_{\sR}}\prod_{i=1}^r\om_{k_i}^{j_i(v_i+\sum_{h=1}^{\ell-1}u_{i,h})}\right|+\left|\sum_{W\in
    {\bar
    \sW_{\sR}}}\prod_{i=1}^r\om_{k_i}^{j_i(v_i+\sum_{h=1}^{\ell-1}u_{i,h})}\right|\\
 \le& \left|\prod_{i=1,i\ne
   t}^r\om_{k_i}^{j_i(v_i+\sum_{h=1}^{\ell-1}\phi_{i,h})}\right|\left|\om_{k_t}^{j_tv_t}
   \right|  \left|\sum_{W\in
    \sW_{\sR}}\om_{k_t}^{j_t\sum_{j=1}^{\ell-1}u_{t,j}}\right|+\sum_{W\in
      {\bar \sW_{\sR}}}1\\
  \le & \left|\sum_{W\in
    \sW_{\sR}}\om_{k_t}^{j_t\sum_{h=1}^{\ell-1}u_{t,h}}\right|+\sum_{W\in
      {\bar \sW_{\sR}}}1.
\end{align*}
Applying condition~2 and Proposition~\ref{prop:aper}, we obtain
 $$\left|\sum_{W\in \sW_{\sR}}\om_{k_t}^{j_t\sum_{h=1}^{\ell-1}u_{t,h}}\right|<
 \sum_{W\in \sW_{\sR}}1.$$
 It follows that
 $$\left|\left(T^{\ell}(\om_{k_1}^{j_1},\ldots,\om_{k_r}^{j_r})
    \right)_{\bu,\bv}\right| <\sum_{W\in \sW_{\sR}}1+\sum_{W\in
      {\bar \sW_{\sR}}}1= \left(T^{\ell}({\bf1})\right)_{\bu,\bv}.$$
Applying item~3 of Proposition~\ref{prop:eigen} to the matrix
$T^{\ell}\left((\om_{k_1}^{j_1},\ldots,\om_{k_r}^{j_r})\right)$, we obtain
$$\rho(T(\om_{k_1}^{j_1},\ldots,\om_{k_r}^{j_r}))^{\ell}<\rho(T({\bf
      1}))^{\ell},$$
and hence
$0\le \theta_2<1$.

Using the Jordan normal form of
$T(\bz)$, it is easy to see that, for each $\bz=(\om_{k_1}^{j_1},\ldots,\om_{k_r}^{j_r})$ with $(j_1,\ldots,j_r) \ne (0,\ldots,0)$, all entries of $T^{a-1}(\bz)$ are of the order $O\left(a^N\rho^a\theta_2^a\right)$, where $N$ is the size of $T(\bz)$. Since $N$ is fixed (i.e., independent of $m$ or $a$), it follows that, for $(j_1,\ldots,j_r) \ne (0,\ldots,0)$,
$$
\hat{C}_m(\om_{k_1}^{j_1},\ldots,\om_{k_r}^{j_r})
 = \start(\om_{k_1}^{j_1},\ldots,\om_{k_r}^{j_r})
    T^{a-1}(\om_{k_1}^{j_1},\ldots,\om_{k_r}^{j_r})
    \term_b(\om_{k_1}^{j_1},\ldots,\om_{k_r}^{j_r})
  = O\left(a^N\theta_2^a\right)\rho^{a}.
$$
Next we estimate $\hat{C}_m({\bf1})$. By Proposition~4(1), $\rho$ is a simple eigenvalue of $T({\bf 1})$ (and  $T^{\top}({\bf 1})$), so we may write the Jordan normal form of  $T({\bf 1})$ in the following block form
$$T({\bf 1})=[\bg~ L]\left[\begin{array}{cc}\rho&{\bf 0}\\{\bf 0}&\Lambda \end{array}\right]\left[\begin{array}{c}\bh\\H\end{array}\right]$$
  such that $\bg$ is a positive eigenvector of $T({\bf 1})$ corresponding to $\rho$,  $\bh^{\top}$ is a positive eigenvector of $T^{\top}({\bf1})$
corresponding to $\rho$,  the matrix $\Lambda$ corresponds to other eigenvalues of $T({\bf 1})$ which are all smaller, in absolute value, than $\rho$, and the matrix $\displaystyle \left[\begin{array}{c}\bh\\H\end{array}\right]$ is the inverse of
$[\bg~ L]$.

Since all entries of $\Lambda^{a-1}$ are of the order $O\left(a^N\rho^a\theta_1^a\right)$, it follows that
\begin{align*}
  \qquad  \hat{C}_m({\bf1})
 &= \start({\bf1})T^{a-1}({\bf1})\term_b({\bf1})\\
  &= (\start({\bf1})\cdot\bg)
    (\bh \cdot\term_b({\bf1}))\rho^{a-1}+
    \start({\bf1})L\Lambda^{a-1}H\term_b({\bf1})\\
&= (\start({\bf1})\cdot\bg)
    (\bh \cdot\term_b({\bf1}))\rho^{a-1}+ O\left(a^N\rho^a\theta_1^a\right).
\end{align*}
Since there is at least one arc from $\eps_s$ to $\sR$ and at least one arc
from $\sR$ to $\SEQ_b(G)$, both $\start({\bf1})$ and $\term_b({\bf1})$ are
non-negative vectors with at least one positive entry. Consequently both $\start({\bf1})\cdot\bg$ and
    $\bh \cdot\term_b({\bf1}))$ are positive.
Now Theorem~2 follows from Proposition~\ref{prop:local} with $\theta$ being any constant satisfying
$\max\{\theta_1^{1/\spn},\theta_2^{1/\spn}\}<\theta<1$, and
\[
A=\frac{1}{|G|}(\start({\bf 1})\cdot\bg)
    (\bh \cdot\term_b({\bf1}))\rho^{-1-b/\spn},~
    B=\rho^{1/\spn}.
\qedhere \]
\end{proof}

\begin{example}[continues=exa:d3]
We again consider the class of compositions over $G$ such that the sum of any
three consecutive parts is nonzero.
To make matrices as convenient as possible, we use a span of $\spn=2$.
We specialize to the case $G = \mathbb{Z}_k$, where $k \geq 4$.
An argument similar to that used below in the proof of Corollary~2 part~4 shows that $D_{\sR}$ is
 strongly connected.
The first condition of Theorem~\ref{thm:local} is satisfied by the loop at the vertex $(1,1)$.
The second condition is satisfied by setting $\mathcal{W}_\mathcal{R}$ to
contain the walks
$(1,1), (1,k-1), (2,1), (k-1,1), (k-1,k-1)$ and
$(1,1), (1,k-1), (2,2), (k-1,1), (k-1,k-1)$.

We further specialize, taking $k=4$,
and in this case, with the help of the Computer Algebra System {\em MAPLE}, we can apply Proposition~2 to derive
an exact formula as well as a simple approximation formula for $c_m(\bs)$.
The elements of $\sR$ are
$(1, 1), (1, 2), (1, 3), (2, 1), (2, 2), (2, 3), (3, 1), (3, 2), (3, 3)$.
Using this ordering, we get
\begin{align*}
\start(z) &= \rvect{z^2, z^3, z^4, z^3, z^4, z^5, z^4, z^5, z^6}\\
\term_0(z) &= \rvect{1,1,1,1,1,1,1,1,1}^\top \\
\term_1(z) &=  [\begin{matrix} z + z^3& z^2 + z^3& z + z^2 + z^3& z^2 + z^3& z
  + z^2 + z^3 \end{matrix} \\
&\qquad\qquad \begin{matrix} z + z^2& z + z^2 + z^3& z + z^2& z + z^3
  \end{matrix}]^\top \\
T(z) &= \left[
\begin{array}{ccccccccc}
 z^2 & 0 & z^4 & 0 & 0 & 0 & z^4 & z^5 & z^6 \\
 0 & 0 & 0 & z^3 & z^4 & z^5 & z^4 & z^5 & 0 \\
 z^2 & z^3 & z^4 & z^3 & z^4 & 0 & z^4 & 0 & z^6 \\
 0 & 0 & 0 & 0 & z^4 & z^5 & z^4 & z^5 & z^6 \\
 0 & z^3 & z^4 & z^3 & z^4 & z^5 & z^4 & z^5 & 0 \\
 z^2 & z^3 & z^4 & z^3 & z^4 & 0 & 0 & 0 & 0 \\
 z^2 & 0 & z^4 & 0 & z^4 & z^5 & z^4 & z^5 & z^6 \\
 0 & z^3 & z^4 & z^3 & z^4 & z^5 & 0 & 0 & 0 \\
 z^2 & z^3 & z^4 & 0 & 0 & 0 & z^4 & 0 & z^6 \\
\end{array}
\right].
\end{align*}

Using {\em MAPLE},  we find that the four matrices $T(1),T(-1),T(i),T(-i)$
 are all diagonlizable and
\begin{align*}
\hC_m(1)&=\frac{3}{2}(1+\sqrt{2})^m+\frac{3}{2}(1-\sqrt{2})^m,\\
\hC_m(-1)&=\hC_m(i)=\hC_m(-i)\\
&=(-1)^mf(r_1)r_1^{m/2}+f(r_2)r_2^{m/2}+(-1)^mf(r_3)r_3^{m/2}+2\Re\left(f(r_4)r_4^{m/2}\right),
\end{align*}
where
$$
f(r)=\frac{r^3-5r^2+6r-1}{r(4r^4-11r^3-13r^2-19r+17)},
$$
$r_1\doteq 3.848862156$, $r_2\doteq 0.4736256091$,  $r_3\doteq 0.3639455409$,  and $r_4\doteq -0.8432166528+0.892341334i$
 are eigenvalues of $T(-1),T(i), T(-i)$ satisfying
\[r^5-3r^4-3r^3-2r^2+4r-1=0.\]
It follows from Proposition~2 that
\begin{align*}
c_m(0)&= \frac{3}{8} (1 + \sqrt{2})^m+\frac{3}{8} (1 - \sqrt{2})^m+\frac{3}{4}(-1)^mf(r_1)r_1^{m/2}+\frac{3}{4}f(r_2)r_2^{m/2}\\
&+\frac{3}{4}(-1)^mf(r_3)r_3^{m/2}+\frac{3}{2}\Re\left(f(r_4)r_4^{m/2}\right)\\
&= \frac{3}{8} (1 + \sqrt{2})^m+\frac{3}{4}(-1)^mf(r_1)r_1^{m/2}+\frac{3}{2}\Re\left(f(r_4)r_4^{m/2}\right)+O\left(r_2^{m/2}\right),\\
c_m(1)&=c_m(2)=c_m(3)\\
&= \frac{3}{8} (1 + \sqrt{2})^m+\frac{3}{8} (1 - \sqrt{2})^m-\frac{1}{4}(-1)^mf(r_1)r_1^{m/2}-\frac{1}{4}f(r_2)r_2^{m/2}\\
&-\frac{1}{4}(-1)^mf(r_3)r_3^{m/2}-\frac{1}{2}\Re\left(f(r_4)r_4^{m/2}\right)\\
&= \frac{3}{8} (1 + \sqrt{2})^m-\frac{1}{4}(-1)^mf(r_1)r_1^{m/2}-\frac{1}{2}\Re\left(f(r_4)r_4^{m/2}\right)+O\left(r_2^{m/2}\right).
\end{align*}
Table~\ref{tab:asympt} shows numerical values for relevant sequences, where  $$a_m(0)= \frac{3}{8} (1 + \sqrt{2})^{m}+\frac{3}{4}(-1)^mf(r_1)r_1^{m/2}+ \frac{3}{2}\Re\left(f(r_2)r_2^{m/2}\right)$$ is evaluated up to one decimal place.

\begin{table}
\centering
\begin{tabular}[c]{|cccc|}
\hline
$m$ & $c_m(1)$& $c_m(0)$ & $a_m(0)$ \\ \hline
2 &  2&3 & 2.7 \\
3 & 7&0 & -0.1 \\
4 & 10&21 & 20.8 \\
5 & 35&18 & 17.9 \\
6 &64& 105 & 104.9 \\
7 & 199&120 & 119.9 \\
8 & 396&543 & 542.9 \\
9 &1119& 822 & 821.9 \\
10 &2376& 2961 & 2960.9 \\
11 & 6373&5238 & 5237.9 \\
12 & 14142&16377 & 16376.9 \\
13 & 36589&32196 & 32195.9 \\
14 & 83532&92133 & 92132.9 \\
15 &211075& 194196 & 194195.9 \\
16 & 491110&524241 & 524240.9 \\
17 &1221885& 1156908 & 1156908.0 \\
18 & 2878806&3006279 & 3006279.0 \\
19 & 7089517&6839406 & 6839406.0 \\
20 & 16841988&17332647 & 17332647.0 \\
21 & 41196941&40234356 &40234356.0\\\hline
\end{tabular}
\caption{Values of $c_m(1)$, $c_m(0)$ and $a_m(0)$ from Example~\ref{exa:d3} }
\label{tab:asympt}
\end{table}
\qedhere
\end{example}

\begin{proof}[Proof of Corollary~\ref{cor:local}]
We first note that the matrix $T({\bf 1})$ is a 0-1 matrix. Condition~1 implies
that each row of $T({\bf 1})$ contains exactly $K$ 1's.
Hence the dominant eigenvalue of
$T({\bf 1})$ is $\rho=K$ and $\rvect{1,1,\cdots, 1}^{\top}$ is a corresponding eigenvector.
Conditions~2 and 3 imply that $\start({\bf 1})$ contains exactly $H$ 1's and
$\term_b({\bf 1})=J\rvect{1,1,\cdots, 1}^{\top}$. Since $\term_b({\bf 1})$ is an eigenvector of $T({\bf 1})$ corresponding to the dominant eigenvalue $\rho$, $\hat{C}_m({\bf 1})$ can be evaluated exactly without using the Jordan normal form. And we have
\[\hat{C}_m({\bf 1})=\start({\bf1})T^{a-1}({\bf 1})\term_b({\bf
1})=\start({\bf1})K^{a-1}J\rvect{1,1,\cdots, 1}^{\top}=H\cdot J\cdot  K^{a-1},
\]
where we used the fact $H=\start({\bf1})\rvect{1,1,\cdots, 1}^{\top}$.
This establishes (\ref{cor1}).
\end{proof}

\begin{proof}[Proof of Corollary~\ref{cor:list}]
We first define the digraph $D$ for each class and compute the values of $H,J$
and $K$.
For part~1, we let $\sR$ be the set of all $(d+1)$-tuples of distinct elements
of $G$.
Hence $\spn=d+1$ and  $|\sR|=|G|^{\underline {d+1}}$.
The vertex $\eps_s$ is joined to all recurrent vertices and so
$H=|G|^{\underline {d+1}}$.
A recurrent vertex $\bu=(u_1,u_2,\ldots,u_{d+1})$ is joined to a recurrent
vertex $\bv=(v_1,v_2,\ldots,v_{d+1})$  if $v_j$ is different from
$v_1,\ldots, v_{j-1},u_{j+1},\ldots, u_{d+1}$ for each $1\le j\le d+1$.
Hence $\bu$ is joined to  $K=(|G|-d)^{d+1}$ recurrent vertices.
Similarly $\bu\in \sR$ is joined to a vertex $\bv=(v_1,v_2,\ldots,v_b)$ if
$v_j$ is different from $v_1,\ldots, v_{j-1},u_{j+1},\ldots, u_{d+1}$ for each
$1\le j\le b$.
Hence $J=(|G|-d)^b$.

For part~2, we let $\sR$ be the set of all $(d+1)$-tuples of distinct elements
of $G^*$. So the remaining argument is exactly the same as for part~1 with $G$
being replaced by $G^*$.

For part~3,  the corresponding digraph $D$ is defined
as in part~1 except
that there is an arc from $\eps_s$ to a vertex $\bv\in \sR$ only if the first
$d$ parts of $\bv$ are all nonzero.
Thus we have $\spn=d+1$,
$J=(|G|-d)^{b}$, $K=(|G|-d)^{d+1}$, and $H=(|G|-1)^{\underline d}(|G|-d)$.

For part~4, we let
$\sR=\{(u_1,\ldots,u_d): u_j\in G,\sum_{j=1}^du_j\ne {\bf 0}\}$.
So we have $\spn=d$.
We note that for every choice of $(u_1,\ldots,u_{d-1})$, there are $|G|-1$
choices of $u_d$ such that the total sum is nonzero (we need the assumption
$d\ge 2$ here).
Hence $|\sR|=(|G|-1)|G|^{d-1}$.
The vertex $\eps_s$ is joined to all recurrent vertices and so
$H=(|G|-1)|G|^{d-1}$.
A recurrent vertex $\bu=(u_1,u_2,\ldots,u_{d})$ is joined to a vertex
$\bv=(v_1,v_2,\ldots,v_k)\in \sR\cup \SEQ_b(G)$ if $v_1+\cdots
+v_{j}+u_{j+1}+\cdots +u_d$ is not zero for each $1\le j\le k$.
Hence $K=(|G|-1)^{d}$ and $J=(|G|-1)^b$.

For part~5, we take $\sR$ to be the set of all $d$-tuples $x_1,\ldots,x_d$ of
nonzero elements of $\mathbb{F}_q$ such that $x_d \neq (x_1 \cdots
x_{d-1})^{-1}$.
If $q=p^n$, each field element is represented as an $n$-tuple over $\Zp$.
This gives $\sigma = d$ and $H = |\sR| = (q-1)^{d-1}(q-2)$.
We have $K = (q-2)^d$ since to join $x_1,\ldots,x_d$ to $y_1,\ldots,y_d$,
each $y_i$ is constrained to be nonzero and not the multiplicative inverse
of the previous $d-1$ parts.
Similarly, $J = (q-2)^b$.\\

Next we verify that the digraph $D_{\sR}$ is strongly connected by showing that
there is a directed walk from any recurrent vertex  to any other recurrent
vertex.

For part~1, Let ${\bf b}=(b_1,b_2,\ldots,b_{d+1})$ and ${\bf a}=(a_1,a_2,\ldots,a_{d+1})$ be any two distinct recurrent vertices.
Let $j$ be the smallest integer such that $a_{j+1}\ne b_{j+1}$. Thus we can write ${\bf a}=(b_1,\ldots,b_j,a_{j+1},\ldots,a_{d+1})$,
 where $0\le j\le d$ and $a_{j+1}\ne b_{j+1}$. We use induction on $d-j$ to show that there is a directed walk from $\bf b$ to $\bf a$. The basis case $j=d$ is obvious
 since ${\bf b}$ is joined to $\bf a$. Now we move to the inductive step by finding a recurrent vertex $(b_1,\ldots,b_j,b_{j+1},x_{j+2},\ldots,x_{d+1})$
which is joined to $\bf a$.  If $b_{j+1}\notin \{a_{j+2},\ldots,a_{d+1}\}$, then $(b_1,\ldots,b_j,b_{j+1},a_{j+2},\ldots,a_{d+1})$ is joined to $\bf a$.
If $b_{j+1}=a_k$ for some $k\in\{j+2,\ldots, d+1\}$, then for any $y\notin \{b_1,\ldots,b_{j+1},a_{j+1},\ldots,a_{d+1}\}$, it is easy to check that
$(b_1,\ldots,b_j,b_{j+1},a_{j+2},\ldots,a_{k-1},y,a_{k+1},\ldots,a_{d+1})$ is joined to $\bf a$.

The argument for parts~2 and 3 is similar to above.

For part~4, we show that there is a directed walk from a recurrent vertex ${\bf b}=(b_1,b_2,\ldots,b_{d})$ to another recurrent vertex
${\bf a}=(b_1,\ldots,b_j,a_{j+1},\ldots,a_{d})$ using induction on $d-j$ as for part~1. When $j=d-1$, it is clear that $\bf b$ is joined to $\bf a$.
For the inductive step, let $y\in G$ satisfy
\begin{align*}
y+(b_1+\cdots+b_{j+1}+a_{j+3}+\cdots+a_d)&\ne 0,\\
y+(b_1+\cdots+b_j+a_{j+1}+a_{j+3}+\cdots+a_d)&\ne 0.
\end{align*}
Then it is easy to check that $(b_1,\ldots,b_j,b_{j+1},y,a_{j+3},\ldots,a_d)$ is joined to $\bf a$.

The argument for part~5 is similar to that for part~4.\\

Finally we verify the two conditions in Theorem~\ref{thm:local} for
the five classes of compositions.
The verification proceeds as follows.
For each $1\le t\le r$, we
find two distinct recurrent vertices $\bu=(u_1,u_2,\ldots,u_{\spn})$ and
$\bv=(v_1,v_2,\ldots,v_{\spn})$ such that $\bu$ is joined to itself and to
$\bv$, $\bv$ is joined to itself, and $|\bv|-|\bu|={\bf e}_t$ where
${\bf e}_t$ is the $t$\textsuperscript{th} unit vector of dimension $r$.
Hence condition~1 is satisfied because of the loop at $\bu$, and
$\sW_{\sR}=\{\bu\bu\bv,\bu\bv\bv\}$ is a set
of directed walks of length 2 from $\bu$ to $\bv$ which satisfies condition~2.

For part~1, we choose $d$ distinct elements $u_1,\ldots,u_d$ from $G\setminus
\{{\bf0},{\bf e}_t\}$ and let $\bu=(u_1,\ldots,u_d,\mathbf{0})$ and
$\bv=(u_1,\ldots,u_d,{\bf e}_t)$.
It is easy to see that there are $(|G|-2)^{\underline d}>0$ choices of
$u_1,\ldots,u_d$.

For part~2, we distinguish two cases.
If $r=1$, then $G=\mZ_k$ for some $k\ge d+3$.
In this case we simply let
$\bu=(3,4,\ldots,d+2,1)$ and $\bv=(3,4,\ldots,d+2,2)$.
If $r\ge 2$, we let $t$ and $t'$ be two distinct integers in
$\{1,2,\ldots,r\}$.
We then choose $d$ distinct elements $u_1,\ldots,u_d$ from
$G\setminus \{{\bf 0},{\bf e}_{t'},{\bf e}_{t'}+{\bf e}_t\}$ and let
$\bu=(u_1,\ldots,u_d,{\bf e}_{t'})$ and
$\bv=(u_1,\ldots,u_d,{\bf e}_{t'}+{\bf e}_t)$.
It is clear that there are $(|G|-3)^{\underline d}>0$ choices of
$u_1,\ldots,u_d$.

For part~3, we may use the same $\bu$ and $\bv$ as in part~1.

For part~4, again we discuss two cases.
The case $G=\mZ_k$ is treated as in part~2.
If $r\ge 2$, we let $t$ and $t'$ be two distinct integers in
$\{1,2,\ldots,r\}$, and choose $d-1$ elements $u_1,\ldots,u_{d-1}$ from $G$
such that $u_1+\ldots+u_{d-1}+{\bf e}_{t'}\ne {\bf 0}$ and
$u_1+\ldots+u_{d-1}+{\bf e}_{t'}+{\bf e}_{t}\ne {\bf0}$.
We then let
$\bu:=(u_1,\ldots,u_{d-1},{\bf e}_{t'}), \bv:=(u_1,\ldots,u_{d-1},{\bf
e}_{t'}+{\bf e}_t)$.
There are $(|G|-2)|G|^{d-2}$ choices of $u_1,\ldots,u_{d-1}$.

Part~5 is handled in a similar manner to part~4.
\end{proof}

\section{Bijections and exact values}\label{sec:biject}
Here we first provide numerical values for the number of locally Mullen
compositions and Carlitz compositions.
Table~\ref{tab:2mul} shows initial values of $c_m(\bs)$ for locally 2-Mullen
compositions;
Table~\ref{tab:w3car} is similar for 2-Carlitz weak compositions; and
Table~\ref{tab:3car} is similar for 2-Carlitz compositions.

\begin{table}
\centering
\begin{tabular}[c]{|c|cccccccccc|}
\hline
\diagbox{$s$}{$m$} & $1$ & $2$ & $3$ & $4$ & $5$ & $6$ & $7$ & $8$ & $9$ & $10$ \\ \hline
0 & 0 & 0 & 12 & 24 & 48 & 204 & 624 & 1680 & 5196 & 16008 \\
1 & 1 & 3 & 6 & 21 & 69 & 192 & 573 & 1767 & 5262 & 15681 \\\hline
\end{tabular}
\caption{Counts of locally 2-Mullen compositions over $\mathbb{Z}_5$.}
\label{tab:2mul}
\end{table}

\begin{table}
\centering
\begin{tabular}[c]{|c|ccccccccc|}
\hline
\diagbox{$s$}{$m$} & $2$ & $3$ & $4$ & $5$ & $6$ & $7$ & $8$ & $9$ & $10$\\ \hline
0 & 4 & 24 & 88 & 320 & 1248 & 5120 & 20728 & 82284 & 326296 \\
1 & 6 & 18 & 72 & 320 & 1284 & 5120 & 20232 & 81738 & 329064 \\
2 & 4 & 18 & 88 & 320 & 1236 & 5120 & 20728 & 81738 & 326296 \\
3 & 6 & 24 & 72 & 320 & 1392 & 5120 & 20232 & 82284 & 329064 \\
4 & 4 & 18 & 88 & 320 & 1236 & 5120 & 20728 & 81738 & 326296 \\
5 & 6 & 18 & 72 & 320 & 1284 & 5120 & 20232 & 81738 & 329064 \\\hline
\end{tabular}
\caption{Counts of 2-Carlitz weak compositions over $\mathbb{Z}_6$.}
\label{tab:w3car}
\end{table}

\begin{table}
\centering
\begin{tabular}[c]{|c|ccccccccc|}
\hline
\diagbox{$s$}{$m$} & $2$ & $3$ & $4$ & $5$ & $6$ & $7$ & $8$ & $9$ & $10$\\ \hline
0 & 4 & 12 & 32 & 80  & 280 & 812 & 2572 & 6644 & 23460 \\
1 & 4 & 6  & 34 & 82  & 284 & 748 & 2498 & 7372 & 21522 \\
2 & 2 & 12 & 32 & 80  & 274 & 866 & 2266 & 7484 & 21642 \\
3 & 4 & 12 & 16 & 136 & 224 & 820 & 2480 & 7384 & 21432 \\
4 & 2 & 12 & 32 & 80  & 274 & 866 & 2266 & 7484 & 21642 \\
5 & 4 & 6  & 34 & 82  & 284 & 748 & 2498 & 7372 & 21522 \\\hline
\end{tabular}
\caption{Counts of 2-Carlitz compositions over $\mathbb{Z}_6$.}
\label{tab:3car}
\end{table}

Next we give bijections between different families of compositions.
The following observation establishes a connection between Carlitz compositions
and  locally $d$-Mullen compositions.

\begin{proposition}\label{prop:bijection}
For each  $m$-composition $\bu=u_1,u_2,\ldots,u_m$ over $G$, let
$\bv=\phi(\bu)$ be an $m$-composition defined by $v_j=u_1+\ldots +u_j, 1\le
j\le m$.
Then $\phi$ is a bijection between locally $d$-Mullen $m$-compositions and
$d$-Carlitz weak $m$-compositions over $G$ such that the first $d$ parts are
nonzero.
\end{proposition}

\begin{proof}
It follows from the definition of the mapping $\phi$ that
$v_{i+j}=v_i+(u_{i+1}+\ldots+u_{i+j})$. It is clear that $v_i\ne v_{i+j}$ if and
only if $u_{i+1}+\ldots+u_{i+j}\ne 0$. It is important to note that $v_k$ might be
zero when $k>d$, however
$v_k\ne 0$ when $1\le k\le d$ when $\bu$ is locally $d$-Mullen.
Also $\bu$ and $\bv$ generally do not have the same sum.
\end{proof}

Using Corollary~2 part~3 and the above proposition, we immediately obtain
the following.
\begin{theorem}\label{theorem:locallyMullen}
Let $G$ be a finite abelian group, and $d$ be a positive integer such that
$|G|\ge d+1$. Let $c_m(\bs)$ be the number of locally $d$-Mullen
$m$-compositions of $\bs\in G$. Then there is a positive constant $\theta<1$
such that
$$c_m(\bs)=\frac{1}{|G|}(|G|-1)^{\underline d}~(|G|-d)^{m-d}\left(1+O(\theta^m)\right),\hbox{ as }m\to \infty.$$
\end{theorem}

\begin{proposition}\label{prop:mullen}

\begin{itemize}
\item[1.] Let $\cal A$ be a family of compositions over a finite
abelian group $G$. Suppose $\cal A$ is closed under multiplication, that is,
if ${\bf x}:=(x_1,\ldots,x_m)$ belongs to $\cal A$ then $a{\bf
x}:=(ax_1,\ldots,ax_m)$ also belongs to $\cal A$ for every $a\in G^*$.
Let  $c_m(\bs; {\cal A})$ be the number of $m$-compositions of $\bs$ in $\cal
A$.
If $\bs\in G^*$ has a multiplicative inverse $\bs^{-1}$ in $G^*$, then
$c_m(\bs;{\cal A})=c_m({\bf 1}; {\cal A})$.
\item[2] Let $c_m(\bs)$ be the number of locally $d$-Mullen $m$-compositions of
$\bs \in G$.
Then $c_m(\bs) = c_m({\bf 1})$ for every $\bs \in G^*$.
\end{itemize}
\end{proposition}

\begin{proof}
1. It is clear that ${\bf x}\mapsto \bs^{-1}{\bf x}$ is a bijection between
compositions of $\bs$ and compositions of $\bf 1$ in $\cal A$.  \\
2.  The claim follows from part~1 if $\bs^{-1}$ exists. For general $\bs\in G^*$, let
$\pi$ be a permutation of $G$ such that $\pi(\bs)={\bf 1}$ and $\pi({\bf
0})={\bf 0}$.
Then  it is easy to verify that  $\phi\pi\phi^{-1}$
is a bijection from locally $d$-Mullen $m$-compositions of $\bs$ to those of
${\bf 1}$.
\end{proof}

{\bf Remark.} The above proposition implies that the number of Carlitz
$m$-compositions of $s$ over a finite field is equal to that of 1 when $s\ne
0$. And the same is true for Carlitz weak compositions.

\section{Conclusion} \label{sec:conclusion}

The structure of finite abelian groups naturally leads us to use
Corollary~\ref{cor:integer} to go from a generating function for restricted
integer compositions to the number of restricted compositions over group
elements.
Assuming some aperiodicity conditions, asymptotic analysis of this expression
gives us the dominant term, as the number of parts goes to infinity.
Exact counts are available as well, by taking powers of the transfer matrix.
A couple of potential extensions to this work are apparent.
It would be interesting to get analogous asymptotic counting results for
non-abelian groups.
A similar level of generality is unlikely but one could be curious about what
techniques are necessary to generalize at least somewhat beyond the restriction
of commutativity.
It may also be interesting to consider further types of restriction such as
pattern avoidance.

\vskip 0.5cm

{\bf Acknowledgement.} We thank E.\ A.\ Bender for indicating that our results
apply to general finite abelian groups not merely to finite fields.

\vskip 20pt

{}

\end{document}